\documentclass{amsart}
\usepackage{amsfonts,amssymb,amsmath,amsthm}
\usepackage{url}
\usepackage{enumerate}
\usepackage{xcolor}

\newcommand{\assign}{:=}

\newcommand{\nin}{\not\in}
\newcommand{\nocomma}{}

\newcommand{\nosymbol}{}
\newcommand{\tmtextit}[1]{{\itshape{#1}}}

\newtheorem{thrm}{Theorem}[section]
\newtheorem{lem}[thrm]{Lemma}
\newtheorem{prop}[thrm]{Proposition}

\theoremstyle{definition}
\newtheorem{definition}[thrm]{Definition}
\newtheorem{remark}[thrm]{Remark}
\numberwithin{equation}{section}

\title{A new generalization of Browder's degree}
\author{Mohammad Niksirat}
\address{Department of Mathematics, University of Toronto, Toronto, Canada, M5S 2E4}
\email{niksirat@math.toronto.edu}
\keywords{degree theory, finite rank approximation, monotone maps}

\subjclass[2010]{Primary 47H11, Secondary 47H07}
\begin{document}

\begin{abstract}
A new generalization of the Browder's degree for the mappings of the type $(S)_+$ is presented. The main idea is rooted in the observation that the Browder's degree remains unchanged for the mappings of the form $A: Y\to X^*$, where $Y$ is a reflexive uniformly convex Banach space continuously embedded in the Banach space $X$. The advantage of the suggested degree lies in the simplicity it provides for the calculations of degree associated to nonlinear operators. An application from the theory of phase transition in liquid crystals is presented for which the suggested degree has been successfully applied.
\end{abstract}
\maketitle

\section{Introduction}
For a Banach space $X$, the continuous pairing between $X^*$, the topological dual of $X$, and $X$ is denoted by $\langle , \rangle$. A map $A : X \rightarrow X^{\ast}$ is called
$(S)_+$ if for every sequence $(u_n) \subset X$, $u_n \rightharpoonup u$, the
inequality
\begin{equation}
  \limsup_{n\to\infty}  \langle A [u_n], u_n - u \rangle \leq 0,
\end{equation}
imply $u_n \to u$; see e.g. \cite{Browder82,Berkovits86}. It is well known ( see for example \cite{Skrypnik1994}) that the map 
\begin{equation}
A[u]:=\sum_{\alpha\leq m} (-1)^{|\alpha|} D^\alpha f_\alpha(x,u,\dots,D^m u),
\end{equation}
from $X= W_0^{m, p} (\Omega)$ to $X^*= W^{- m, q}(\Omega)$ is $(S)_+$ where the pairing $\langle A[u],\phi	\rangle$ is defined by the relation
\begin{equation}
  \langle A [u], \phi \rangle = \sum_{| \alpha | \leq m} \int_{\Omega}
  f_{\alpha} (x, u,\dots,D^m u) D^{\alpha} \phi.
\end{equation}
The existence and the multiplicity problems of the quasi-linear elliptic equations in divergence form can be studied by the degree of the operator $A$ at zero.

A degree theory, keeping all classical properties of a topological degree, for the bounded demi-continuous $(S)_+$ maps in Hilbert spaces is developed by I. Skrypnik \cite{Skrypnik73}. F. Browder generalized the degree to the mapping in uniformly convex reflexive Banach spaces \cite{Browder82,Browder1976}.  Browder's construction is based on the direct generalization of the classical Brouwer degree through the Galerkin type approximation of $(S)_+$ mappings. An alternative construction through generalizing the Leray-Schauder degree has been
carried out by J. Berkovits \cite{Berkovits1997,Berkovits86}. A new generalization of the Browder's degree theory from the Nagumo degree is reported by A. Kartsatos and D. Kerr \cite{Karts11}. 

A turning point in the application of degree theory is made by Y.Y. Li  \cite{Li1989}, where the author uses the Fitzpatrick's degree of quasi-linear Fredholm maps to define a degree for fully nonlinear second order elliptic equations. A remarkable progress is reported by I. Skrypnik \cite{Skrypnik1994}, showing that every fully nonlinear uniformly
elliptic equation (satisfying some growth rate conditions) is represented by an operator equation involving $(S)_+$ mapping. This formulation opens up a way to study the well-posedness problem of such equations by topological degree argument.

Nevertheless, in all constructions of degree, some type of continuity (and in the weakest case, the demi-continuity) is required and can not be relaxed. A map $A:X\to X^*$ is called demi-continuous if the strong convergence $x_n\stackrel{X}\to x$ implies the weak convergence $A x_n \stackrel{X^*}\rightharpoonup A x$. However, this assumption fails for some real applications. Here we give an example from the phase transition in liquid crystals. It is shown that (see \cite{Niksirat14a,Niksirat14}), the stationary solution of the Doi-Onsager equation
\begin{equation}
  \label{dyn} \frac{\partial f}{\partial t} = \Delta_r f +  {\rm div}
  (f \nabla_r U (f)),
\end{equation}
for the interaction potential function $U(r)$
\begin{equation}
  \label{U2} U (f) (r) = \lambda \int_{S^2} |r\times r'| f (r') d \sigma (r'),
\end{equation}
and the probability density function $f(r)$ of the directions of the rod-like molecules 
\begin{equation}
  \label{f2} f (r) = \left( \int_{S^2} e^{- U (f)} d \sigma \right)^{- 1} e^{-
  U (f) (r)} ,
\end{equation}
reduces to the fixed point problem $u-\lambda \Gamma[u]=0$, where
\begin{equation}
  \label{Onsag} \Gamma [u] (r) := \left( \int_{S^2} e^{- u (r)} \right)^{- 1}
  \int_{S^2} \left (|r\times r'|-\frac{\pi}{4}\right) e^{- u (r')} d \sigma (r').
\end{equation}
The natural function space for the problem is
\begin{equation}
H_0(S^2)=\{u\in L^2(S^2),u(-r)=u(r),\int_{S^2}{u(r) d\sigma(r)}=0 \}.
\end{equation}
It is simply seen that $\Gamma$ fails to be demi-continuous in any open neighbourhood of $0 \in H_0(S^2)$. Fix $r \in S^2$, and let $u_n$ be the following sequence
\begin{equation}
  u_n (r) = \left\{ \begin{array}{ll}
    \log (2 \pi (1 - \cos (1 / n))) & \cos^{- 1} (r. \bar{r})^{\nosymbol} \in
    \left( 0, \frac{1}{n} \right)\\
    0 & {\rm otherwise}
  \end{array} \right. .
\end{equation}
Obviously, $u_n \xrightarrow{H_0 (S^2)} 0$, while
\[ \lim_n \Gamma [u_n] = \lim_n  \frac{1}{2 \pi (1 - \cos (1 / n))}  \int_0^{1 / n}
   \int_0^{2 \pi} \hat{K} (\gamma) d \sigma \neq G (0) = 0. \]
   We have the following theorem.
\begin{thrm}[Niksirat \cite{Niksirat14a}]
  \label{Hw} Fix $\lambda$ and let $\Omega_\lambda$ be the following set
  \[
  \Omega_\lambda= \{u \in H_0 (S^2), |u (r) | \leq
    \lambda  \} .
  \]
  The map $\Gamma : \Omega_\lambda \subset H_0(S^2) \to H_0(S^2)$ is continuous and compact.
 \end{thrm}
By the above theorem, the map $\Gamma:W^{k,2}(S^2)\cap H_0(S^2)\to H_0(S^2)$ for $k>3$ is continuous due to the compact embedding of $W^{k,2}$ into the space of continuous functions. For the problem defined on the unit circle, the author \cite{Niksirat14} used the classical degree theory of the map 
\[
\mathbb{I}-\lambda G:W^{k,2}(S^1)\cap H_0(S^1)\to \left( W^{k,2}(S^1)\cap H_0(S^1)\right)^*.
\]
For higher dimensional problems, the calculations are intricate and lengthy. In this paper, we construct a new degree theory for the mapping $A : Y\subset X\to X^*$ where $X$ is a
Banach space, $Y$ is a separable reflexive Banach space continuously embedded
in $X$, and $A$ a bounded, demi-continuous and $(S)_+$. The constructed degree enjoys all properties of a classical topological degree. The class of homotopies for which the suggested degree remain constant is narrower than the degree usually that is usually defined for $(S)_+$ mappings.

\section{Definition of the degree}
Let $A: Y \to X^*$ be a map where $X$ is a Banach spaces and $Y$ is a separable reflexive Banach space continuously embedded
in $X$. Without loss of generality, we can assume that $Y$ is a reflexive separable locally uniformly convex space due to the Kadec-Klee theorem (see \cite{Fabian11} Theorem 8.1).
On the other hand, if $Y$ is a separable locally uniformly Banach space, the Browder-Ton embedding theorem, see \cite{Browder1968a,Berkovits2003}, implies the existence of a Hilbert spacee $H$ such that the embedding $j:H\hookrightarrow Y$ is dense and compact. Choosing a basis $\mathcal{H}=\{h_1,h_2,\dots\}$ for $H$, we can define the set $\mathcal{Y}=\{y_1,y_2,\dots\}$ where $y_k=j(h_k)$, and accordingly, the filtration $\{Y_1,Y_2,\dots\}$ for $Y$ where $Y_n={\rm span}\{y_1,\dots,y_n\}$.
\begin{definition}
For $A:Y\to X^*$, the finite rank approximation $A_n: Y \rightarrow Y_n$ is defined by the relation 
\begin{equation}\label{FRA}
  A_n (u) = \sum_{k = 1}^n \langle A [u], i (y_k) \rangle y_k,
\end{equation}
where $\langle, \rangle$ denotes the continuous pairing between $X^{\ast}$ and
$X$ and $i : Y \to X$ denotes the continuous embedding of $Y$ into
$X$. We should notice that the above approximation is completely different from the finite rank approximation of the map $i^* A:Y \to Y^*$.
\end{definition}
The inner product $(,)$ in $Y_n$ is uniquely defined by the relation $(y_i,y_j)=\delta_{i j}$. It is simply verified that $A_n$ and $A$ coincide on $Y_n$ in the following sense
\begin{eqnarray}
  \langle A [u], i(v) \rangle = (A_n (u), v),\hspace{0.4cm} \forall v\in Y_n.
\end{eqnarray}
Let $W$ be a subspace of $X$. We write $A [u]\stackrel{W} =0$ if $\langle A [u], i(v) \rangle= 0$ for all $v \in W$. In sequel, we assume that $\Omega\subset Y$ is an open and bounded set, and $A : \Omega
\to X^*$ is bounded, demi-continuous and $(S)_+$ in the following sense:
\begin{itemize}
  \item $A$ is bounded if $A [\Omega]$ is bounded in $X^{\ast}$.
  
  \item $A$ is demi-continuous if $u_n \stackrel{Y}\to u$, then $A
  [u_n] \stackrel{X^*}\rightharpoonup A [u]$. In our setting, the latter weak convergence reads 
  $\langle A [u_n], i (y) \rangle \rightarrow \langle A [u], i
  (y) \rangle$ for all $y \in Y$. Note that this notion is weaker than the
  demi-continuity of a map from $X$ to $X^*$.
  
  \item $A$ is $(S)_+$ if for every sequence $(u_n) \subset Y$, $u_n
  \rightharpoonup u $, the inequality
  \begin{equation}
    \limsup_{n \rightarrow \infty}  \langle A [u_n], i(u_n - u) \rangle \leq
    0,
  \end{equation}
  imply $u_n \rightarrow u$. 
\end{itemize}
We further assume that $A$ satisfies the following condition $(H)$:
\begin{eqnarray}
  (H) : \{ u \in Y ; A[u]\stackrel{Y} = 0 \} = \{ u
  \in Y ; A [u] \stackrel{X}= 0 \} .
\end{eqnarray}
\begin{lem}
  \label{0-lem}Let $D \subset \bar{\Omega}$ be a closed set. If there exists a
  sequence $(u_n) \subset D\cap Y_n$ such that $A[u_n]\stackrel{Y_n} = 0$, 
  then the equation $A [u] = 0 \in X^*$ is solvable in $D$. In
  particular, if $u_n \in \partial \Omega$ and $A [u_n] \stackrel{Y_n}= 0$ then
  there is $u \in \partial \Omega$ such that $A [u] \stackrel{X}= 0$.
\end{lem}

\begin{proof}
  Let $(u_n)\subset D$ be a sequence and $A[u_n]\stackrel{Y_n} = 0$. Since $D$ is bounded, there exists $u \in Y$ and a subsequence of $(u_n)$
  (that we still denote it by $u_n$) such that $u_n \rightharpoonup u$. Take an arbitrary
  sequence $(\zeta_n)$, $\zeta_n \in Y_n$ such that $\zeta_n \stackrel{Y}\longrightarrow
  u$. By the relation $A[u_n]\stackrel{Y_n}=0$ and the boundedness property of $A$, we have
  \begin{eqnarray}
    \limsup_{n \to \infty} \langle A [u_n], i (u_n - u) \rangle &
    = & \limsup_{n \to \infty}  \langle A [u_n], i
    (\zeta_n - u) \rangle \nonumber\\
    & \leq & \limsup_{n \to \infty} \| A [u_n]
    \|_{X^*}  \| \zeta_n - u \|_Y = 0. 
    \label{eq:201507021058}
  \end{eqnarray}
 Since $A$ is of class $(S)_+$ on $\Omega$, the inequality (\ref{eq:201507021058}) and the closedness of $D$ imply $u_n \to u\in D$.
  By the demi-continuity of $A$ on $Y$ we conclude $A [u_n] \rightharpoonup A [u]$. Now let $y \in Y$ be arbitrary. Take a sequence $(\xi_n)$, $\xi_n \in
  Y_n$ such that $\xi_n \xrightarrow{Y} y$. We have
  \begin{eqnarray}
    \langle A [u], i (y) \rangle & = & \lim_{n \to \infty} 
    \langle A [u_n], i (y) \rangle \\
    & = & \lim_{n \to \infty} \langle A [u_n], i (y -
    \xi_n) \rangle \nonumber\\
    & \leq & \lim_{n \to \infty}  \| A [u_n] \|_{X^*}
    \| y - \xi_n \|_{Y} = 0. 
  \end{eqnarray}
  Replacing $y$ by $- y$ implies $\langle A[u_n],i(y) \rangle \geq 0$ and thus $A [u]\stackrel{Y}= 0$. Finally the condition $(H)$ implies $A [u] \stackrel{X}= 0$. 
\end{proof}

\begin{definition} \label{deg-def}
Assume that $0 \not\in A [\partial\Omega]$. The index of $A$ in $\Omega$ at $0 \in X^{\ast}$ is defined by the relation
\begin{eqnarray}
  \label{deg-A-An} {\rm ind} (A, \Omega) = \lim_{n \to \infty}
  \deg_B (A_n, \Omega_n, 0),
\end{eqnarray}
where $\Omega_n = \Omega \cap Y_n$ and $\deg_B$ denotes the Brouwer degree.
\end{definition}
\remark{
  The index of $A$ is different from the Browder degree of the map
  $A : Y \rightarrow Y^*$ even if $Y$ is a subspace of
  $X$. For example, let $X$ be a uniformly convex Banach space $X=Y\oplus \{u\}$ and $J:X \to X^*$ the duality map. Consider the map
  $A:Y\to X^*$, $\langle A[x],y+t u\rangle=\langle J(x),y\rangle+t $. It is simply seen that the Browder degree of the map $A:Y\to Y^*$ equals
  $1$ but the index of the map $A:Y\to X^*$ is $0$ by the definition (\ref{deg-def}). We first justify the definition (\ref{deg-def}).}

\begin{prop}
  Assume that $0 \not\in A [\partial\Omega]$. The Brouwer degree of $A_n$ in
  $\Omega_n$ is stable, that is, there exists $N_0 > 0$ such that for $n \geq
  N_0$ we have
  \begin{equation}
    \label{deg-stbl} \deg_B (A_{n - 1}, \Omega_{n - 1}, 0) = \deg_B (A_n,
    \Omega_n, 0) .
  \end{equation}
\end{prop}

\begin{proof}
  We first show that there is $N_0 > 0$ such that $0 \notin A_n (\partial
  \Omega_n)$ for all $n \geq N_0$. Assuming the contrary, let there exist a sequence
  $u_n \in \partial \Omega_n$ such that $A [u_n] \stackrel{Y_n}= 0$.
  As $\Omega$ is open in $Y$, we have
  $u_n \subseteq \partial \Omega$. An argument similar to the one employed in the proof of the lemma (\ref{0-lem}), implies the
  existence of $u \in \partial \Omega$ such that $A [u] = 0$, a contradiction! 
  Hence, there is $N_0 > 0$ such that the Brouwer degree of $A_n$ in $\Omega_n$
  is well defined at $0$ for $n \geq N_0$. Consider the map $B_n : \Omega_n \rightarrow
  Y_n$ defined by
  \begin{eqnarray}
    B_n (u) \assign (A_{n - 1} (u), Pr_n (u) ),
  \end{eqnarray}
  where $Pr_n (u)$ denotes the $n$-th component of $u$ in $Y_n$. Apparently, we have
  \begin{eqnarray}
    \label{An-Bn} \deg_B (A_{n - 1}, \Omega_{n - 1}, 0) = \deg_B (B_n,
    \Omega_n, 0) .
  \end{eqnarray}
  Note that the proof ends once we show that for $n$ large enough there holds
  \begin{eqnarray*}
    \deg_B (A_n, \Omega_n, 0) = \deg_B (B_n, \Omega_n, 0) .
  \end{eqnarray*}
  Consider the following convex homotopy $h (t)$:
  \begin{eqnarray}
    \label{h-hom} h_n (t) = (1 - t) A_n + t B_n .
  \end{eqnarray}
  It suffices to show that $h_n (t) (u) \neq 0$ for all $u \in \partial \Omega_n$ and for all
  $t \in [0, 1]$ when $n$ is sufficiently large. 
  Assume the contrary. Then there exists $z_n \in \partial \Omega_n$ and $t_n
  \in [0, 1]$ such that $h (t_n) (z_n) \stackrel{Y_n}= 0$. By the definition of finite rank approximation (\ref{FRA}),
  we can write $A_n(z_n)$ as
  \[
  A_n(z_n)=\left(A_{n-1}(z_n),\langle A[z_n],i(y_n) \rangle y_n \right),
  \]
  where $\{y_1,y_2,..\}$ is a frame for $Y$. Now, the relation $h(t_n)(z_n)\stackrel{Y_n}=0$ implies that
  $A[z_n]\stackrel{Y_{n-1}}=0$ and
  \begin{equation}
  (1 - t_n) \langle A [z_n], i (y_n) \rangle y_n  + t_n Pr_n (z_n)\stackrel{Y_n} = 0. \label{eq:201507021403}
  \end{equation}
  But $A[z_n]\stackrel{Y_{n-1}}=0$ implies $A_n (z_n) = r_n y_n$ for
  some $r_n \in \mathbb{R}$. This together with (\ref{eq:201507021403}) gives
  $r_n = \frac{- t_n}{1 - t_n} z_{n,n},$ where $z_{n,n}$ is the component of $z_n$ along $y_n$, that is, $Pr_n (z_n) = z_{n,n} y_n$. Therefore, we have
  \begin{equation}
    \langle A [z_n], i (z_n) \rangle = (A_n (z_n), z_n) = - \frac{t_n}{1 -
    t_n} z_{n,n}^2 \leq 0.\label{znn}
  \end{equation}
  Since $\partial \Omega$ is bounded and $(z_n)\in \partial\Omega$, there is a subsequence $(z_{n_k})$ such that
  $z_{n_k} \rightharpoonup z \in Y$. To keep notations simple, we denote this subsequence by $(z_n)$. Let $\zeta_n \in Y_{n - 1}$ be a sequence such that $\zeta_n \stackrel{Y}\to z$. We have
  \begin{eqnarray}
    \langle A [z_n], i (z_n - z) \rangle & = & \langle A [z_n], i (z_n -
    \zeta_n) \rangle + \langle A [z_n], i (\zeta_n - z) \rangle \nonumber\\
    & \leq & \langle A [z_n], i (z_n) \rangle + \| A [z_n] \|  \|
    \zeta_n - z \| \nonumber\\
    & \leq & \| A [z_n] \|  \| \zeta_n - z \| . 
  \end{eqnarray}
 The last equality follows from (\ref{znn}). Therefore
  \begin{eqnarray}
    \limsup_n  \langle A [z_n], i (z_n - z) \rangle \leq 0,
  \end{eqnarray}
  and since $A$ is $(S)_+$, we conclude $z_n \to z \in
  \partial \Omega$. Finally, following the lemma (\ref{0-lem}), we reach $A [z] = 0$,
  a contradiction!
\end{proof}

\section{Properties of the invariant}

We show that the topological invariant defined in the definition (\ref{deg-def}) satisfies the classical
properties of a topological degree. First we identify the admissible class of
homotopy.

\begin{definition}[Admissible homotopy]
  \label{hom-def}Let $\Omega$ be an open bounded subset of $Y$. A one parameter family of maps $h : [0, 1] \times \Omega
  \to X^*$ is called an admissible homotopy if it satisfies the following conditions:
  \begin{enumerate}
    \item $h : [0, 1] \times \Omega \rightarrow X^{\ast}$ is bounded and
    demi-continuous,
    
    \item $0 \not\in h ([0, 1] \times \partial \Omega)$,
    
    \item for every sequence $t_n \rightarrow t$, $t_n \in [0, 1]$ and $u_n
    \rightharpoonup u$ for $u_n \in \bar{\Omega}$, the inequality
    \[
      \limsup_{n \to \infty}  \langle h (t_n) (u_n), i (u_n - u)
      \rangle \leq 0,
    \]
    implies $u_n \to u$,
    
    \item for all $t \in [0, 1]$ we have
    
      \[
      H (t) : \{ z \in Y ; h (t) (z) \stackrel{Y}= 0 \} =
      \{ z \in Y ; h (t) (z) \stackrel{X}= 0 \} .
    \]
  \end{enumerate}
\end{definition}

\begin{thrm}
  \label{deg-pro} Let $A : \Omega \rightarrow X^{\ast}$ be a bounded,
  demi-continuous, $(S)_+$ mapping that satisfies the condition $(H)$ and furthermore,
  $0 \not\in A [\partial \Omega]$. The index of $A : \Omega
  \subset Y \to X^*$ given in the definition(\ref{deg-def}) has the following properties:
  \begin{enumerate}
    \item the equation $A [u] = 0$ is solvable in $\Omega$ if ${\rm ind} (A,
    \Omega) \neq 0$,
    
    \item if $\Omega = \Omega_1 \cup \Omega_2$ where $\Omega_1$ and $\Omega_2$ are disjoint open sets
    then
    \[
      \label{D-D} {\rm ind} (A, \Omega) = {\rm ind} (A,
      \Omega_1) + {\rm ind} (A, \Omega_2),
    \]
    \item If $h : [0, 1] \times \Omega \rightarrow X^{\ast}$ is an admissible
    homotopy then ${\rm ind} (h (t), \Omega)$ is independent of $t \in [0,
    1]$.
  \end{enumerate}
\end{thrm}

\begin{proof}
  To prove (1), let us assume ${\rm ind} (A, \Omega) =1$. Therefore, there is $N$ such that ${\rm deg}_B(A_n,\Omega_n,0)=1$ for all $n\geq N$. By the properties of Brouwer degree, the equation $A_n (u) = 0\in Y_n$ is solvable in $\Omega_n$. Let $(u_n)\subset
    \overline{\Omega}_n$ be a sequence such that $A_n (u_n) = 0$. Now lemma (\ref{0-lem}) implies the existence of some
     $u \in \overline{\Omega}_n$ such that $A [u] \stackrel{X}= 0$. Since $0 \not\in A [\partial \Omega]$, we
    conclude $u \in \Omega$. The second property also follows directly from the domain decomposition of the Brouwer degree. In fact, if $\Omega=\Omega_1 \cup \Omega_2$ then for sufficiently large $n$ we have
\[
{\rm ind}(A,\Omega)={\rm deg}_B(A_n,\Omega_{1,n}\cup \Omega_{2,n},0),
\]  
where $\Omega_{1,n}=\Omega_1\cap Y_n$  and $\Omega_{2,n}=\Omega_2\cap Y_n$.  By the domain decomposition of  the Brouwer degree, we have
\[
{\rm deg}_B(A_n,\Omega_{1,n}\cup \Omega_{2,n},0)={\rm deg}_B(A_n,\Omega_{1,n},0)+{\rm deg}_B(A_n, \Omega_{2,n},0).
\]
Now the claim is proved due to the relation 
\[
{\rm deg}_B(A_n,\Omega_{k,n},0)={\rm ind}(A,\Omega_k), k=1,2.
\]    
 To prove the homotopy invariance property of the defined index, assume
\begin{equation}
    {\rm ind} (h (t),\Omega) \neq {\rm ind} (h (s),\Omega),
\end{equation}
for some $s,t\in[0,1]$. Let $n$ be so large such that the following relations hold
    \begin{eqnarray}
      \label{hom-1} {\rm ind} (h (t),\Omega) = \deg_B (h_n (t),
      \Omega_n, 0),
    \end{eqnarray}
    \begin{eqnarray}
      \label{hom-2} {\rm ind} (h (s),\Omega) = \deg_B (h_n(s),
      \Omega_n, 0),
    \end{eqnarray}
   Therefore, there exist $\tau_n \in [s, t]$ and $\zeta_n \in \partial \Omega$
    such that $h (\tau_n)  (\zeta_n) \stackrel{Y_n}= 0$. Since $(\zeta_n)$ is
    bounded, by an argument similar to the proof of the lemma (\ref{0-lem}), we
    derive the inequality
    \begin{eqnarray}
      \limsup_{n \rightarrow \infty}  \langle h (\tau_n) (\zeta_n), i (\zeta_n
      - \zeta) \rangle \leq 0,
    \end{eqnarray}
    for some $\zeta \in Y$. The condition $(3)$ in the
    definition (\ref{hom-def}) implies $\zeta_n \to \zeta \in
    \partial \Omega$. Therefore, $h (\tau) (\zeta) \stackrel{Y}= 0$ for some $\tau $, a limit point of the sequence
    $(\tau_n)\subset [s, t]$. The condition $H(t)$ 
    implies $h (\tau) (\zeta) \stackrel{X}= 0$ which contradicts
    the condition $(2)$ of the definition (\ref{hom-def}). Therefore there exist $N_0 > 0$ such
    that 
    \begin{eqnarray}
      {\rm ind} (h_n (t), \Omega) = {\rm ind} (h_n (s), \Omega), \forall n\geq N_0.
    \end{eqnarray}
   and thus the index is independent of the homotopies in the class of admissible homotopy defined in the definition (\ref{hom-def}).   
\end{proof}

\section{The stationary solutions of Doi-Onsager equation in $\mathbb{R}^2$}
Here we study the one dimensional Doi-Onsager equation, and employ the suggested degree in this article. As we have shown earlier in \cite{Niksirat14}, this problem is reduced to the fixed point problem  
\begin{equation}
A[u]:=u-\lambda \Gamma[u]
\end{equation}
where
\begin{equation}
\Gamma[u]=\left(\int_{0}^{2\pi}e^{-u(\theta)}d⁣\theta\right)^{-1}\int_0^{2\pi}\hat{K}(\theta-\theta') e^{-u(\theta')} d⁣\theta'.
\end{equation}
The kernel $\hat{K}$ is assumed to be in $ W^{1,\infty}[0,2\pi])$ having the following expansion
\begin{equation}
\hat{K}(\theta)=\sum_{n=1}^\infty{k_n \cos(2 n \theta)},
\end{equation}
where $k_n<0$ for all $n$ and they satisfy the condition
$
k_1<k_2<k_3<\cdots.
$
It is obviously seen that the Onsager's kernel $\hat{K}(\theta)=|\sin(\theta)|-\frac{2}{\pi}$ satisfies the above assumptions. Let $Y$ be the space
\[
Y=\left\{u\in W^{1,2}([0,2\pi]),u(\theta)=u(\pi+\theta) a.e.,u(\theta)=u(2\pi-\theta)  a.e.,\int_0^{2\pi}u(\theta)d⁣\theta=0\right\}.
\]
It is seen that $A:Y\to Y$ is a continuous map in the class of $(S)_+$ mappings. In order to simplify calculations for the degree argument, we employ the constructed degree above and study the map $A:Y\to X$ where $X$ is the Hilbert space
\[
X=\left\{u\in L^2([0,2\pi]),u(\theta)=u(\pi+\theta) a.e.,u(\theta)=u(2\pi-\theta)  a.e.,\int_0^{2\pi}u(\theta)d⁣\theta=0\right\}.
\]

\begin{thrm}[Niksirat \cite{Niksirat14}]
Assume that $\hat{K}\in W^{1,\infty}([0,2\pi])$. Then $\Gamma:Y\to Y$ is continuous and compact.
\end{thrm}

The following proposition guarantees that all the solutions of $A[u]=0$ are in the space $Y$.

\begin{prop}
Assume $u\in X$ and $A[u]=0$, then $u\in Y$.
\end{prop}
\begin{proof}
We have
\begin{equation}
u(\theta)=\lambda \left(\int_{0}^{2\pi}e^{-u(\theta)}d⁣\theta\right)^{-1} \int_0^{2\pi} \hat{K}(\theta-\theta') e^{-u(\theta')} d\theta'.
\end{equation}
We have
\begin{equation}
||\frac{d}{d\theta} u||^2=\lambda^2 \left(\int_{0}^{2\pi}e^{-u(\theta)}d⁣\theta\right)^{-2} 
\int_0^{2\pi}\left[\int_0^{2\pi}\frac{d}{d\theta} \hat{K}(\theta-\theta') e^{-u(\theta')} d\theta'\right]^2 d\theta.
\end{equation}
Since $\hat{K}\in W^{1,\infty}([0,2\pi])$, then
\begin{equation}
||\frac{d}{d\theta} u|| \leq 2\pi \lambda ||\frac{d}{d\theta} \hat{K}||_\infty< \infty
\end{equation}
\end{proof}

\begin{lem}
For the map $A:Y\to Y^*$, let the finite rank approximation $\tilde{A}_n:Y\to Y_n$ be defined by the relation
\begin{equation}
\tilde{A}_n(u)=\sum_{k=1}^n{ \langle A[u],y_k \rangle_Y  y_k},
\end{equation}
and let $\Omega\subset Y$ be an open bounded subset. If $0\nin A(\partial\Omega)$, then we have the following relation
\begin{equation}
\lim_{n\to\infty} \deg(\tilde{A}_n, \Omega_n,0)=\lim_{n\to\infty} \deg(A_n,\Omega_n,0),
\end{equation}
where $A_n$ is the finite rank approximation of $A:Y\to X$ defined in Def.(\ref{FRA}), and $\Omega_n:=\Omega\cap Y_n$.
\end{lem}

\begin{proof} Otherwise there is a sequence $z_n\in \partial\Omega_n$, and a sequence $t_n\in (0,1)$ such that
\[
t_n \tilde{A}_n(z_n)+(1-t_n)A_n(z_n)=0.
\]
Since $\Omega$ is bounded, there is a subsequence (we still denote as $z_n$) such that $z_n \rightharpoonup z$. 
For any $y \in Y_n$ we have
\[ \langle A [ z_n], y \rangle_Y = - \frac{1 - t_n}{t_n} \langle A [ z_n], y
   \rangle_X . \]
Since $Y \hookrightarrow C ( [ 0, 2 \pi])$ is a compact embedding, then there
is a subsequence (we still denote as $z_n$) such that $z_n \xrightarrow{X} z$,
and thus
\begin{eqnarray*}
 \limsup_{n \rightarrow \infty} \langle A [ z_n], y \rangle_Y = - \frac{1 -
   \tau}{\tau} \langle A [ z], y \rangle_X, 
\end{eqnarray*}
for some $\tau \in [ 0, 1]$. Let $\zeta_n \in Y_n$ such that $\zeta_n
\xrightarrow{Y} z$, and thus
\begin{eqnarray*}
\limsup_{n \rightarrow \infty} \langle A [ z_n], z_n - z \rangle_Y =
   \limsup_{n \rightarrow \infty} \langle A [ z_n], z_n - \zeta_n \rangle_Y =\\=
   - \frac{1 - \tau}{\tau} \limsup_{n \rightarrow \infty} \langle A [ z_n],
   z_n - z \rangle_X \leq 0.
   \end{eqnarray*}
On the other hand, $\Gamma$ is a compact mapping and therefore
\[ \limsup_{n \rightarrow \infty} \langle A [ z_n], z_n - z \rangle_Y =
   \limsup_{n \rightarrow \infty} \langle z_n, z_n - z \rangle_Y \leq 0, \]
that implies $z_n \xrightarrow{Y} z \in \partial \Omega$.
But for $y_n= \cos(2 n \theta)$ we have
\[
\langle A[z],y_n\rangle_Y=(1+4n^2) \langle A[z],y_n\rangle_X
\]
that holds only if $A[z]=0$ and thus $A[z]=0$. Since $z\in \partial\Omega$, this is a contradiction.
\end{proof}

Now we state the main result of the solution of the equation $A[u]=0$.

\begin{thrm}\label{M-Thm}
The equation $u-\lambda \Gamma[u]=0$ has the unique trivial solution $u\equiv 0$ for 
\[
0<\lambda<\lambda_0:=||\hat{K}||_\infty^{-1},
\] 
and two solutions bifurcate from the trivial solution at $\lambda_n:=-\frac{2}{k_n}$ for $n \geq 1$. The trivial solution is stable for $\lambda< \lambda_1$, and is unstable for $\lambda>\lambda_1$.
\end{thrm}
For the proof of theorem, we need the following Gr\"{u}ss type of inequality that the reader can see \cite{Dragomir} for a proof.

\begin{lem}
  \label{G-lem}Assume that $d \mu$ is a probability measure and $f, g \in
  L_{\infty} ( \Omega)$, then
  \[ \left| \int_{\Omega} f d \mu \int g d \nocomma \mu - \int f g d \mu
     \right| \leq \| f \|_{\infty} \| g \|_{\infty} \]
\end{lem}

\begin{lem}
  \label{Isol}If $\lambda < \| \hat{K} \|_{\infty}^{- 1}$, then all solutions
  to the equation $A [ u] = 0$ are isolated.
\end{lem}

\begin{proof}
  According to the Fredholm alternative theorem it is enough to show that $\ker ( Id - \lambda D \Gamma) [ u] = \{ 0
  \}$. We calculate the Jacobian matrix $D \Gamma$. A straightforward
  calculation gives the following formula for the G{\^a}teaux derivative
  \begin{equation}
    \label{DG} D \Gamma [ u] ( v) = \int_0^{2 \pi} v ( \theta) d \mu_u
    \int_0^{2 \pi} \hat{K} ( \theta - \theta') d \nocomma \mu_u - \int_0^{2
    \pi} \hat{K} ( \theta - \theta') v ( \theta') d \mu_u,
  \end{equation}
  where $d \mu_u$ denotes the probability measure
  \[ d \mu ( \theta) = \left( \int_0^{2 \pi} e^{- u ( \theta)} d \nocomma
     \theta \right)^{- 1} e^{- u ( \theta)} . \]
  Assume that $v = \lambda D \Gamma [ u] ( v)$. By lemma (\ref{G-lem}), we
  obtain
  \[ | v | = \lambda \left| \int_0^{2 \pi} v ( \theta) d \mu_u \int_0^{2 \pi}
     \hat{K} ( \theta - \theta') d \nocomma \mu_u - \int_0^{2 \pi} \hat{K} (
     \theta - \theta') v ( \theta') d \mu_u \right| \leq \lambda \| \hat{K}
     \|_{\infty} \| v \|_{\infty}, \]
  and thus $v = 0$.
\end{proof}

\begin{proof}(of theorem \ref{M-Thm})
The apriori estimate of the set containing the solution of $A [ u] = 0$ is
\begin{equation}
  \| u \|_{\infty} \leq \lambda \| \hat{K} \|_{\infty},
\end{equation}
and thus there is an open set $\Omega \subset Y$ such that for each $\lambda$,
there is no solution to $A [ u] = 0$ on $\Omega^c$. Therefore, by the homotopy
invariance property of degree for the convex homotopy $h_t = Id - t
\lambda \Gamma$, we have
\[ \deg ( {Id} - \lambda \Gamma, \Omega, 0) = \deg ( {Id}, \Omega,
   0) = 1. \]
Now, we calculate the Jacobian matrix of the map $\Gamma$ in the basis
$\left\{ \phi_n \assign \frac{1}{\sqrt{\pi}} \cos ( 2 \nocomma n \theta)
\right\}$, that is, $a_{n \nocomma m} ( u) \assign \langle D \Gamma [ u] (
\phi_n), \phi_m \rangle$. According to (\ref{DG}), we derive
\[ a_{n \nocomma m} = \int_0^{2 \pi} \phi_n ( \theta) d \mu \int_0^{2 \pi}
   \int_0^{2 \pi} \hat{K} ( \theta - \theta') \phi_m ( \theta) d \nocomma \mu
   d \theta - \int_0^{2 \pi} \int_0^{2 \pi} \hat{K} ( \theta - \theta') \phi_n
   ( \theta') \phi_m ( \theta) d \mu d \nocomma \theta, \]
that gives
\[ a_{n \nocomma m} = k_m \left\{ \int_0^{2 \pi} \cos ( 2 \nocomma n \theta) d
   \mu \int_0^{2 \pi} \cos ( 2 \nocomma m \nocomma \theta) d \nocomma \mu -
   \int_0^{2 \pi} \cos ( 2 n \theta) \cos ( 2 m \theta) d \mu \right\}, \]
and thus by Lemma (\ref{G-lem}) we obtain $| a_{n \nocomma m} | \leq | k_m |$.
On the other hand, since the solutions for $\lambda < \| \hat{K}
\|_{\infty}^{- 1}$ are isolated due to Lemma (\ref{Isol}), the index of any
solution is
\[ {\rm ind} ( u, 0) = {\rm sign} \det ( {Id} - \lambda ( a_{n \nocomma
   m})) = 1, \]
due to the calculation $| a_{n \nocomma m} | < | k_m |$ and the assumption
$\lambda < \| \hat{K} \|_{\infty}^{- 1}$. Since the trivial solution $u \equiv
0$ has index $1$ and the degree of $A$ is $1$, we conclude that $u \equiv 0$
is the unique solution to the equation $A [ u] = 0$ if $\lambda < \| \hat{K}
\|_{\infty}^{- 1}$. Now we show that two solutions bifurcate at $\lambda_n = -
2 k_n^{- 1}$ from the trivial solution $u \equiv 0$ for each $n \geq 1$. Note
that $a_{n \nocomma m} ( 0) = - \frac{k_n}{2} \delta_{n \nocomma m}$, and thus
\[ A_0 : = ( a_{n  m} ( 0)) = {\rm diag} \left( - \frac{k_n}{2}
   \right), \]
a diagonal matrix. Therefore $\det ( {Id} - \lambda A_0)$ changes its
signs at $\lambda_n = 2 k_n^{- 1}$. It is seen that the dimensions of maps $\ker (
{Id} - \lambda_n D \Gamma [ 0])$ and $( {\rm Ran} ( {Id} -
\lambda_n D \Gamma [ 0]))^{\bot}$ are equal to $1$ and thus ${Id} -
\lambda_n D \Gamma [ 0]$ is a Fredholm map of index $0$. The equation $A [ u]
= 0$ can be rewritten as
\[ F ( \lambda, u) : = u - \lambda D \Gamma [ 0] ( u) + G ( \lambda, u) = 0,
\]
where $G ( \lambda, u) = \Gamma [ u] - D \Gamma [ 0] ( u)$. It is simply seen
that $G ( \lambda, u) = o ( \| u \|)$ uniformly in $\lambda$. Now, by Theorem
(28.3) in {\cite{Deimling}}, we conclude that $\lambda_n$ is a bifurcation
point for the equation $A[u]=0$ with two bifurcation solutions of the form
\[ F^{- 1} ( 0) = \{ ( \lambda_n + \mu ( t), t \nocomma v + t z ( t)), -
   \delta < t < \delta \}, \]
for some $\delta > 0$ where $v \in {\rm Ker} ( {Id} -\lambda_n D \Gamma [
0])$ and $z ( t)$ lies in a complement of ${\rm span} \{ v \}$. For the stability, we just notice that the linear operator  $Id-\lambda D\Gamma[0]$ which is the linearization of $A[u]$ at the trivial solution $u\equiv 0$ is unstable for $\lambda>\lambda_1$, and is stable for $0<\lambda<\lambda_1$.
\end{proof}

\end{document}